\documentclass[a4paper]{amsart}

\newcommand{\sptitle}{Sieved Enumeration of Interval Orders and Other Fishburn Structures}
\newcommand{\ptitle}{\sptitle}

\usepackage{bm} 
\usepackage{multirow}
\usepackage[T1]{fontenc}
\usepackage{lmodern}
\usepackage{amssymb,amsmath}
\usepackage{color}
\usepackage[utf8]{inputenc}
\usepackage[unicode=true]{hyperref}
\usepackage[hyperref=true,
            url=false,
            doi=false,
            isbn=false,
            backref=true,
            maxcitenames=3,
            maxbibnames=100,
            block=none]{biblatex}
\hypersetup{
  breaklinks=true,
  bookmarks=true,
  pdfauthor={Stuart Alexander Hannah},
  pdftitle={\ptitle},
  colorlinks=true,
  citecolor=blue,
  urlcolor=blue,
  linkcolor=blue,
  pdfborder={0 0 0}}
\newtheorem{theorem}{Theorem}

\newtheorem{proposition}[theorem]{Proposition}
\newtheorem{corollary}[theorem]{Corollary}

\newtheorem{question}[theorem]{Question}
\theoremstyle{definition}
\newtheorem{definition}[theorem]{Definition}
\newtheorem{remark}[theorem]{Remark}
\newtheorem{example}[theorem]{Example}

\usepackage{tikz}
\usetikzlibrary{matrix,arrows,patterns}
\newcommand{\nodestyle}{
  \tikzstyle{every node} = [font=\small];
}
\newcommand{\discstyle}{
  \tikzstyle{wht} =
    [ circle,fill=white,draw=black, minimum size=4.5pt, inner sep=0pt ];
}
\newcommand{\style}{
  \nodestyle
  \discstyle
}
\newcommand{\pattern}[4]{
  \raisebox{0.6ex}{
  \begin{tikzpicture}
    [scale=0.35, baseline=(current bounding box.center), #1]
    \foreach \x/\y in {#4}
      \fill[pattern=north east lines, pattern color=black!45] 
      (\x,\y) rectangle +(1,1);
    \draw (0.01,0.01) grid (#2+0.99,#2+0.99);
    \foreach \x/\y in {#3}
      \filldraw (\x,\y) circle (5pt);
  \end{tikzpicture}}\;
}                                                                                                                                              

\newcommand{\etal}{et al.\ }
\newcommand{\F}{\mathcal{F}}
\newcommand{\sym}{\mathfrak{S}}

\newcommand{\meshpata}{\sigma }

\newcommand{\qfact}[2]{\bm{({#1})_{#2}!}}

\DeclareMathOperator{\asc}{asc}

\DeclareMathOperator{\Pre}{Pred}
\DeclareMathOperator{\pre}{pred}
\DeclareMathOperator{\Suc}{Succ}
\DeclareMathOperator{\suc}{succ}

\newcommand{\mislabeling}[1]{mislabeling#1}
\newcommand{\mislabelings}{\mislabeling{s}}

\newcommand{\red}{\color{red}}
\newcommand{\blue}{\color{cyan}} 

\newcommand{\oeis}[1]{\href{http://oeis.org/#1}{#1}}

\begin{document}
  \title[\sptitle]{\ptitle}
  \author{Stuart A. Hannah}
  \address{Computer and Information Sciences\\ Livingstone
    Tower\\University of Strathclyde\\Glasgow\\Scotland}
  \date{\today}
  \begin{abstract}
    Following a result of Eriksen and Sj\"{o}strand (2014) we detail a
    technique to construct structures following the Fishburn
    distribution from appropriate Mahonian structures.

    This technique is introduced on a bivincular pattern of
    Bousquet-M\'elou \etal~(2010) and then used to introduce a
    previously unconsidered class of matchings; explicitly, zero
    alignment matchings according to the number of arcs which are both
    right-crossed and left nesting.

    We then define a statistic on the factorial posets of Claesson and
    Linusson~(2011) counting the number of features which we refer to
    as \mislabelings{} and demonstrate that according to the number of
    \mislabelings{} that factorial posets follow the Fishburn
    distribution.

    As a consequence of our approach we find an identity for the
    Fishburn numbers in terms of the Mahonian numbers.
  \end{abstract}

  \maketitle

    \section{Introduction}
  Eriksen and Sj\"{o}strand~\cite{equidistributed} provide a
  remarkable refinement of Zagier's formula for certain Fishburn
  structures. It follows from their results that the following
  patterns are equidistributed in permutations:
  \begin{center}\label{pattern-intro}
  \pattern{}{3}{1/2,2/3,3/1}{1/0,1/1,1/2,1/3,0/1,2/1,3/1},
  \pattern{}{3}{1/3,2/2,3/1}{1/0,1/1,1/2,1/3,0/1,2/1,3/1}
  and 
   \pattern{}{3}{1/1,2/3,3/2}{1/0,1/1,1/2,1/3,0/1,2/1,3/1}.
  \end{center}
  Furthermore they prove that the above patterns are equidistributed
  with the number of right nestings in non-left nesting
  matchings. They show the distribution is given by the coefficients
  $f_{n,k}$ of the following ordinary generating function,
  \begin{align*}
    \sum_{n=0} \sum_{\pi \in \sym_n} x^n y^{\meshpata({\pi})}
    & =  \sum_{n=0} \sum_{k=0} f_{n,k} x^n y^k \\
    &= \sum_{m=0} (-1)^m \prod_{i=1}^m \frac{(1+ (y-1)x)^i-1}{1-y}.
  \end{align*}

  We refer to this as the \emph{Fishburn distribution}.

  Observing that the above equation can be written in terms of a
  substitution of the $q$-factorial $\qfact{n}{q}$ as
  \begin{align*}\label{main-eq}
    \sum_{n \geq 0} \qfact{n}{x(y-1)+1}x^n,
  \end{align*}
  allows for a new approach for the enumeration of Fishburn
  structures. This makes explicit a link between Fishburn and Mahonian
  structures.  In this paper we demonstrate this technique on three
  structures.  We first recreate Eriksen and Sj\"{o}strand's result on
  occurrences in permtuations of the first of the above patterns.  We
  then introduce a new Fishburn structure showing that the number of
  arcs which are both left-nesting and right-crossed in matchings with
  no alignments follow the Fishburn distribution.  Then we identify a
  new feature on the factorial posets of Claesson and
  Linusson~\cite{factorial-matchings} which we name \mislabelings{}.
  A factorial poset with zero \mislabelings{} satisfies a condition of
  Claesson and Linusson giving that the poset is a canonically labeled
  interval order, thus allowing us to recreate a result of
  Bousquet-M\'elou \etal~\cite{two-plus-two-free}.

  As a consequence of the relationship to the $q$-factorial we 
  provide a new identity for coefficients $f_{n,k}$ of the Fishburn
  distribution with respect to the Mahonian numbers $m_{n,k}$
  (\oeis{A008302}), that
  $$ f_{n,k}= \sum_{i=k}^{n-2} (-1)^{i+k}{i \choose k} \sum_{j=i}^{n -i
     \choose 2} {j \choose i} m_{n-i,j}.$$

  Of particular interest is when $k=0$, which gives an identity for
  the $n$th Fishburn number (\oeis{A022493})
  $$\sum_{i=0}^{n-2} (-1)^{i} \sum_{j=i}^{n -i \choose 2} {j
    \choose i} m_{n-i,j}.$$

  \newcommand{\stat}{\psi}
\newcommand{\feature}{feature}
\newcommand{\features}{\feature{}s}
\section{Terminology and background}

For $a,b \in \mathbb{Z}$ with $a < b$ let $[b]$ denote the set $\{1,
\dots, b \}$ and $[a,b]$ the set $\{a, \dots, b \}$. 

For $U$ a linearly ordered set with $x \in U$ where $x$ is \emph{not}
the maximal element of $U$ then, where there is no ambiguity, we shall
abuse notation using $x+1$ to refer to the immediate successor of $x$
in $U$.

\subsection{Mahonian numbers}
For $n \in \mathbb{N}$ let $\qfact{n}{q}$ denote the $q$-factorial,
defined as
\begin{align*}
  \qfact{n}{q} &= \prod_{i=1}^n \sum_{j = 0}^{i-1} q^j 
  = \prod_{i=1}^n \frac{1-q^i}{1-q}.
\end{align*}

The coefficients of the $q$-factorial are known as the Mahonian
numbers (\oeis{A008302}). The first few terms are shown in Figure
\ref{fig-mahonian}. We shall use $m_{n,k}$ to denote the $k$th entry
of row $n$.

Mahonian numbers derive their name from seminal work identifying
permutation statistics by Major MacMahon~\cite{macmahon}. As a result,
and particularly in the case of permutations, structures counted by
the $q$-factorial are often referred to as Mahonian structures.

\begin{figure}
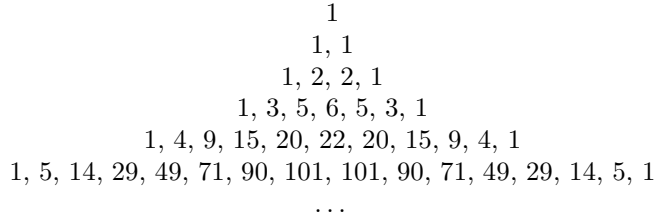

  \begin{center}
    1\\
    1, 1\\
    1, 2, 2, 1\\
    1, 3, 5, 6, 5, 3, 1\\
    1, 4, 9, 15, 20, 22, 20, 15, 9, 4, 1\\
    1, 5, 14, 29, 49, 71, 90, 101, 101, 90, 71, 49, 29, 14, 5, 1\\
    \dots \\
  \end{center}
  \caption{Mahonian Triangle (\oeis{A008302}), $m_{n,k}$ }
  \label{fig-mahonian}
\end{figure}

\subsection{Permutation patterns}

A permutation is a bijection on a finite set $U$. The results in this
paper assume that there is a total order on $U$. We shall therefore
assume throughout that for $n \in \mathbb{N}$ permutations as elements
of $\sym_n$ are bijections on the set $[n]$.

For $n,k \in \mathbb{N}$ with $n > k$ take permutations $\pi \in
\sym_n$ and $\tau \in \sym_k$.  An occurrence of $\tau$ as a
\emph{classical permutation pattern} in $\pi$ is a subsequence of
$\pi$ whose entries are in the same relative order as in $\tau$. For
example taking $\tau = 132$ and $\pi =4671253$ then the following
subsequences of $\pi$ correspond to occurrences of $\tau$,
$$465 \quad 475 \quad 153 \quad 253.$$

Permutations may be represented on a grid by dots placed at line
intersects such that each line is intersected by exactly one dot. The
permutation maps the value indicated by each vertical line to the
value of the corresponding horizontal line indicated by the dot
placement.  For example, the grid below corresponds to the permutation
$4671253$.
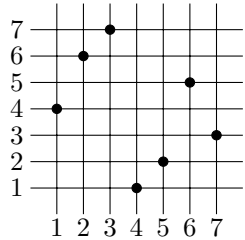
\begin{center}
  \begin{tikzpicture}[scale=0.35, baseline=(current bounding box.center)]
    \draw (0.01,0.01) grid (7+0.99,7+0.99);
    \foreach \x in {1,2,3,4,5,6,7}{
      \draw (\x,-0.5) node {\x};
      \draw (-0.5,\x) node {\x};}
    \foreach \x [count=\i] in {4,6,7,1,2,5,3}
    \filldraw (\i,\x) circle (5pt);
  \end{tikzpicture}
\end{center}

A \emph{mesh pattern}, introduced by Br\"and\'en and Claesson
\cite{meshpatterns}, consists of a classical permutation pattern and
a (potentially empty) set of shaded boxes on the grid representation
of that pattern. An occurrence of a mesh pattern consists of an
occurrence of the underlying classical permutation such that there are
no entries of $\pi$ contained within the shaded boxes.

For example, there are two occurrences of the following mesh pattern in
$\pi$,
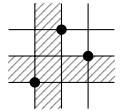
$$\pattern{}{3}{1/1,2/3,3/2}{1/0,1/1,1/2,1/3,0/1,2/1,3/1}.$$ namely
$465$ and $253$.  Whereas, although an occurrence of the underlying
classical pattern, $475$ is not an occurrence the above mesh pattern
as $5$ occurs between the $4$ and $7$.

A \emph{vincular pattern} is a mesh pattern where only entire columns
may be shaded out.  A \emph{bivincular pattern} is a mesh pattern
where any shaded boxes must contribute to an entire row or column of
shaded boxes. The above mesh pattern is also a bivincular pattern.

A permutation with no occurrences of a pattern is said to
\emph{avoid} that pattern.

Occurrences of the pattern 
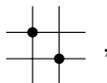
\begin{center}
\pattern{}{2}{1/2,2/1}{},
\end{center}
are known as \emph{inversions} and are counted by the $q$-factorial
(see MacMahon~\cite{macmahon}).

\subsection{Posets}
A poset $P$ is defined as a set and an
associated binary relation $<_P$ satisfying reflexivity,
antisymmetry, and transitivity.  A poset constructed on some
linearly ordered set $U$ is said to be \emph{naturally labeled} if
$i<_P j \implies i<_U j$.

An \emph{interval order} is a poset $P$ where each $z \in P$ can be
assigned a closed interval $[l_z,r_z]\in \mathbb{R}$ such that $x<_P
y$ if and only if $r_x < l_y$. Equivalent conditions are that an
interval order is a poset whose predecessor sets can be assigned a total
order by inclusion~\cite{bogart} or that a poset is an interval
order if it has no induced subposet isomorphic to the pair of
disjoint two element chains, i.e.~the poset is
$(2+2)$-free~\cite{fishburn}.

For $i \in P$, let the following notation be used for the
predecessor and successor sets of $i$:
\begin{align*}
  \Pre{i} = \{ j \in P : j <_P i\}, \quad \quad &\pre{i} = | \Pre{i}|, \\
  \Suc{i}= \{ \ell \in P : i <_P \ell \}, \quad \quad &\suc{i} 
  = | \Suc{i}|.
\end{align*}

\subsection{Matchings}

A perfect matching of size $n$ is a fixed point free involution of
semi-length $n$. Matchings are typically represented as a set of
ordered pairs $(i,j)$ such that $i<j$ or diagrammatically as arcs on
the numberline $[2n]$. For example, the matching of size $10$
$$\{(1,10),(2,9),(3,6),(4,11),(5,7),(8,12),(13,15),(14,16)\},$$

is diagrammatically represented as

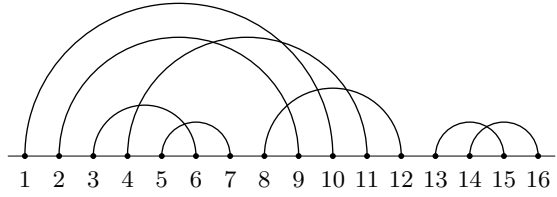
\begin{center} 
\begin{tikzpicture}[line width=.5pt,scale=0.45, font=\small]
 \filldraw (1,0) circle (1.9pt); \node at (1,-0.7){1};
 \filldraw (2,0) circle (1.9pt); \node at (2,-0.7){2};
 \filldraw (3,0) circle (1.9pt); \node at (3,-0.7){3};
 \filldraw (4,0) circle (1.9pt); \node at (4,-0.7){4};
 \filldraw (5,0) circle (1.9pt); \node at (5,-0.7){5};
 \filldraw (6,0) circle (1.9pt); \node at (6,-0.7){6};
 \filldraw (7,0) circle (1.9pt); \node at (7,-0.7){7};
 \filldraw (8,0) circle (1.9pt); \node at (8,-0.7){8};
 \filldraw (9,0) circle (1.9pt); \node at (9,-0.7){9};
 \filldraw (10,0) circle (1.9pt); \node at (10,-0.7){10};
 \filldraw (11,0) circle (1.9pt); \node at (11,-0.7){11};
 \filldraw (12,0) circle (1.9pt); \node at (12,-0.7){12};
 \filldraw (13,0) circle (1.9pt); \node at (13,-0.7){13};
 \filldraw (14,0) circle (1.9pt); \node at (14,-0.7){14};
 \filldraw (15,0) circle (1.9pt); \node at (15,-0.7){15};
 \filldraw (16,0) circle (1.9pt); \node at (16,-0.7){16};
\draw[very thin] (0.5,0) -- (16.500000,0);
 \draw (10,0) arc (0:180:4.5);
 \draw (9,0) arc (0:180:3.5);
 \draw (6,0) arc (0:180:1.5);
 \draw (11,0) arc (0:180:3.5);
 \draw (7,0) arc (0:180:1.0);
 \draw (12,0) arc (0:180:2.0);
 \draw (15,0) arc (0:180:1.0);
 \draw (16,0) arc (0:180:1.0);
\end{tikzpicture}
\end{center}

A \emph{nesting} arc in a matching is an arc which is entirely
encloses another arc when seen diagrammatically, i.e.\ an $(i,j)$ such
that there exists $(k,\ell)$ with $i<k<\ell<j$. The arc which is
enclosed is known as a \emph{nested} arc. If $k=i+1$ then the arcs are
called \emph{left-nesting} and \emph{left-nested} respectively. If
$\ell+1 = j$ then the arcs are \emph{right-nesting} and
\emph{right-nested} respectively.

A \emph{crossing} arc in a matching is the leftmost of two
intersecting arcs when seen diagrammatically, i.e.\ an $(i,j)$ such that
there exists $(k,\ell)$ with $i<k<j<\ell$. The same approach as for
nestings is taken to define \emph{crossed, left-crossing,
  left-crossed, right-crossing} and \emph{right-crossed} arcs.

An alignment in a matching is two arcs $(i,j)$ and $(k,\ell)$ such
that $i < j < k < \ell$. 

For two arcs $(i,j)$ and $(k,\ell)$, we say that $k$ is an
\emph{embraced nested opener} if $k$ is the opener for an arc nested
by $(i,j)$.

\subsection{Statistics and features}

Given some set of structures $X$ a \emph{statistic} $\stat$ is defined
as a function taking a structure to a positive integer, i.e.\ $\psi: X \to
\mathbb{Z}_+$.

A \emph{\feature{}} of a structure is a property, aspect or
substructure of a combinatorial structure. For example, an inversion
in a permutation, or a nesting in a matching.

  \subsection{Background}

Studying interval orders and permutations avoiding $\meshpata$
Bousquet-M\'elou \etal~\cite{two-plus-two-free} demonstrate that their
ordinary generating function is given by
$$\sum_{n \geq 0} \prod_{k=1}^n (1-(1-x)^k) = \qfact{n}{-x+1}x^n,$$ a
function originally considered by Zagier~\cite{zagier-chords} in
enumerating a restricted class matchings (non-neighbor nesting
matchings).

As an intermediate structure Bousquet-M\'elou \etal~ introduce
\emph{ascent sequences}, sequences of the form $b_1 b_2 \dots b_n$
defined recursively with $b_1 =0$ and $b_{i+1} \in [0,\asc{(b_1 b_2
    \dots b_i)} +1]$ where $\asc{}$ counts the total number of ascents
contained within the sequence. Ascent sequences are used to encode the
construction of interval orders via insertions of new maximal
elements. Bousquet-M\'elou \etal present a case analysis to determine
the predecessor set of the $i$th inserted element which is both
dependent on the value $b_i$ and elements previously inserted. Ascent
sequences are additionally used to construct non-neighbor nesting
matchings and permutaions avoiding $\meshpata$, thus giving bijective
correspondences between the structures.

In studying non-neighbor nesting matchings Zagier gave an asymptoic
forumla for their number $f_n$.
$$f_n \sim  n! \frac{12 \sqrt{3}}{\pi^{\frac{5}{2}}}
e^{\tfrac{\pi^2}{12}} \left( \frac{6}{\pi^2} \right )^n \sqrt{n}.$$

We refer to $f_n$ as the \emph{Fishburn numbers}. 

Taking advantage of the equivalent definition for interval orders,
that predecessor sets can be given a total order under inclusion,
Dukes, Jel\'inek and Kubitzke~\cite{comp-matrices-spec} provide an
intuitive relation between interval orders and their generating
function.  They show a simple bijection between interval orders and
the integer matrices of Dukes and
Parviainen~\cite{integer_traingular_matrices}.  An integer matrix of
size $n$ is defined as an $m \times m$, upper triangular, non-row and
non-column empty matrix with integer entries summing to $n$. Dukes,
Jel\'inek and Kubitzke show that two elements are related in an
interval order if they share a hook under the diagonal of the integer
matrix.

Claesson and Linusson~\cite{factorial-matchings} introduce the $n!$
matchings, a subset of matchings with no left-nestings enumerated by
$n!$ according to semi-length. Levande~\cite{levande-two-interp},
solving a conjecture of Claesson and Linusson, finds an additional
subset of non left-nesting matchings counted by the Fishburn numbers.

Eriksen and Sj\"{o}strand~\cite{equidistributed} provide bijections
between various Fishburn structures enumerated by $n!$---including
non-left matchings and permutations---and a class of filled partition
shapes.  In doing so they find the full Fishburn distribution for
these structures, differing from  previous work which had focused
solely on avoidance.


  \section{Original Fishburn permutation}

We lead with a previously studied example. Recall the mesh pattern
\begin{center}
$\meshpata$ = \pattern{}{3}{1/2,2/3,3/1}{1/0,1/1,1/2,1/3,0/1,2/1,3/1},
\end{center}

with avoidance originally given by Bousquet-M\'{e}lou
\etal~\cite{two-plus-two-free} and the full distribution given by
Eriksen and Sj\"ostrand~\cite{equidistributed}.

In their paper Eriksen and Sj\"ostrand show a bijection between
permutations and filled partition shapes by using the filled entries
in the partition shapes to encode the insertion of elements into an
ordered list of blocks. Upon completion the block structure is dropped
and the elements read left-to-right return the permutation.  Their
bijection allows that multiple statistics are equidistributed between
the two structures and through this they provide the non-commutative
generating function with respect to those statistics.

In this section we shall focus on a small part of their work by
considering the distribution of occurrences of $\meshpata$ in
isolation from other statistics.  This differs from the work of
Eriksen and Sj\"ostrand in that the proof is based on insertion of
entries into a permutation rather than encoding the construction. Our
application of the sieve principle is the same.

We begin with the fact that the number of inversions in permutations
follow the Mahonian distribution.  To construct a permutation of size
$n$ with $i$ marked occurrences of $\meshpata$ take a permutation of
size $n-i$ with $i$ marked inversions. Each marked inversion will be
used to insert a new entry which is the first entry of an
occurrence of $\meshpata$.  The sieve principle will be then be
applied to return those permutations strictly satisfying that
\emph{all} occurrences of $\meshpata$ are marked.

Define an order on inversions based on the position in the permutation
of the first entry in the tuple and value of the second entry in the
tuple.  For $a_1 a_2 \dots a_n$ a permutation let $(a_i,a_j)$ and
$(a_{i'}, a_{j'})$ be two, non-equal inversions. If $i = i'$ it
follows $a_i = a_{i'}$ and without loss of generality we can assume
$a_j < a_{j'}$. We then define
$$ (a_i, a_j) < (a_{i'}, a_{j'}).$$ 

Otherwise $i \neq i'$ then without loss of generality assume $i <
i'$. Then we define
$$ (a_i, a_j) < (a_{i'}, a_{j'}).$$

As an example, in the permutation $246531$ the following inversions
are sorted 
$$(4,1)<(6,1)<(6,5).$$

According to the above order each inversion $(a_j, a_k)$ is used to
insert a new entry into the permutation. Taking the position before
the leftmost entry to be position $0$, increment all $a_i > a_k $ by
one and insert $a_k + 1$ at position $j-1$. Thus an occurrence of
$\meshpata$ is created.
  
\begin{example}
  Take the permutation $2{\red 4}653{\red 1}$ where we consider
  the following inversions to be marked 
  $$(4,1)<(6,1)<(6,5).$$

  As the values in the inversions change, at each step the next
  inversion to be used will be colored in red. Inserted entries will
  be marked blue.

  The inversion $(4,1)$ is the first inversion under our defined
  order. Increase all entries greater than $1$ by $1$
  $$357641,$$
  and insert $2$ at position $1$
  $$3{\blue 2}5{\red 7}64{\red 1}.$$

  The next inversion is now labeled $(7,1)$ with the $7$ at position
  $4$. Increase all entries greater than $1$ and insert $2$ at
  position $3$
  $$4{\blue 3}6 {\blue 2} {\red 87}51.$$

  Applying the process to the final inversion, now labeled $(8,7)$,
  leads to the permutation
  $$4{\blue 3}6{\blue 2}{\blue 8}9751.$$

  Note that the inserted entries (marked blue) are all the first
  entries in an occurrence of $\meshpata$.
\end{example}

\begin{proposition}\label{proposition-permutation-bijection}
  The above procedure describes a bijection between permutations with
  marked inversions marked and permutations with the first entry in an
  occurrence of $\meshpata$ marked.
\end{proposition}
\begin{proof}
  To show that this mapping is well defined we need to demonstrate
  that at each step the insertion of an entry does not remove an
  occurrence of $\meshpata$ previously inserted by this process.

 This is enforced by the ordering defined on inversions. Let
  $(a_i,a_j)$ and $(a_{i'}, a_{j'})$ be inversions. 
  \begin{enumerate}

  \item If $i = i'$ then $a_j < a_{j'}$ and therefore $ (a_i, a_j) <
    (a_{i'},a_{j'})$. Our insertion process gives that $a_{j'}+1$ is
    inserted in the position immediately following that of $a_j$ and
    thus forming an ascent with $a_j$. Furthermore as $a_{j'}>a_j$,
    the minimal entry in the occurrence of $\meshpata$ that $a_j$ is
    contained in is not incremented. Therefore the occurrence is
    preserved with the newly inserted entry $a_{j'}$ taking the role
    of the largest entry in the occurrence.

  \item If $i< i'$ then $ (a_i, a_j) < (a_{i'}, a_{j'})$. As $a_{j'}+1$ is
    inserted further to the right in the permutation the ascent that
    $a_i$ involved cannot be broken. If $a_{j} < a_{j'}$ the minimal
    entry in the occurrence of $\meshpata$ remains unchanged. If
    $a_{j} > a_{j'}$ then all entries in the occurrence of $\meshpata$
    containing $a_{j}$ are incremented.  If $a_{j} = a_{j'}$ then
    $a_{j'}$ replaces the minimal entry of the occurrence of
    $\meshpata$ containing $a_{j}$.
  \end{enumerate}

  Thus the mapping is well defined.
  To show that the mapping is a bijection we demonstrate that it is
  both injective and surjective. 

  Injectivity is enforced by the total order on inversions and that an
  inversion pair uniquely determines the entry which is inserted.

  For surjectivity note that the process we have defined inserts the
  first entries of marked occurrences of $\meshpata$ in a
  left-to-right order within the permutation. We can consider the
  reverse of the insertion operation taking a permutation with marked
  occurrences of $\meshpata$ to a permutation with marked inversions.

  Given a permutation of size $n$ with marked occurrences of
  $\meshpata$, take the rightmost marked occurrence. Removing the
  first entry contained in the occurrence and standardizing the
  permutation leaves a permutation of size $n-1$ with a marked
  inversion. Surjectivity follows from repeated application.
\end{proof}

\begin{corollary}
  Permutations with marked occurrences of $\meshpata$ 
  are given by the ordinary generating function
  $$u(x,z) = \sum_{n \geq 0} \qfact{n}{xz+1} x^n.$$
\end{corollary}
\begin{proof}
  Permutations with respect to inversions are enumerated by the
  $q$-factorial. Under the above process an inversion is either
  marked, in which case a new entry uniquely specifying a marked
  occurrence of $\meshpata$ is inserted, or it is unmarked. This is
  equivalent to the substitution $(xz+1)$ in place of $q$ in the
  $q$-factorial with the marking of the occurrence of $\meshpata$
  denoted by $z$.
\end{proof}

 Recreating Eriksen and Sj\"{o}strand's result we now apply the sieve
 principle to permutations with \emph{subsets} of $\meshpata$ marked
 returning those with \emph{all} $\meshpata$ marked. For more details
 see Wilf~\cite[Chapter 4, Section 2]{generatingfunctionology}.

\begin{corollary}
  Permutations with respect to occurrences of $\meshpata$ are given by
  the ordinary generating function
  $$\sum_{n \geq 0} \qfact{n}{x(y-1)+1} x^n.$$
\end{corollary}
\begin{proof}
  The previous corollary gives that permutations with respect to
  marked occurrences of $\meshpata$ are given by the ordinary
  generating function
  $$u(x,z) =\sum_{n \geq 0} \qfact{n}{xz+1}x^n.$$

  In this set a permutation with $k$ marked occurrences of $\meshpata$
  occurs a total of $k \choose i$ times with $i$ occurrences of
  $\meshpata$ marked.

  Let $f(x,y)$ be the ordinary generating function for permutations
  with \emph{all} occurrences of $\meshpata$ marked. Consider the
  substitution of $y$ by $z+1$. This corresponds to remarking
  occurrences of $\meshpata$ with a $z$, or unmarking them them with
  the $1$. As such each permutation will occur $k \choose i$ times
  with $i$ occurrences of $\meshpata$ now marked by $z$. We then have
  that
  $$u(x,z) = f(x,z+1).$$
    
  The result then follows through the reverse substitution of $z$ by
  $y-1$ into $u(x,z)$.
\end{proof}

\begin{remark}
  The distributions of
 \begin{center}
   \pattern{}{3}{1/3,2/2,3/1}{1/0,1/1,1/2,1/3,0/1,2/1,3/1}
   and
   \pattern{}{3}{1/1,2/3,3/2}{1/0,1/1,1/2,1/3,0/1,2/1,3/1}
  \end{center}
 given by Eriksen and Sj\"{o}strand~\cite{equidistributed} can be
 shown in a near identical manner. Again the key is to note that in
 occurrences of these patterns each has a point whose value and
 position are uniquely determined by the other points and that
 together these other two points form an inversion.
\end{remark}

  \section{Technique}

We can generalize the previous two corollaries to explicitly state a 
new technique for constructing Fishburn structures. We present it as
the following theorem.

\begin{theorem}\label{theorem-q-factorial-construction}      
  Let $\F$ be a Mahonian stucture according to the distribution of
  some $q$-\feature{}.
  
  $\F$ follows the Fishburn distribution with respect to some feature
  $p$ if we can show there is a bijection between $\F$ structures of
  size $n$ with $i$ marked $q$-\features{} and $\F$ structures of size
  $n+i$ with $i$ marked $p$-\features{}.
\end{theorem}
\begin{proof}
  By definition, the distribution of $q$-features in $\F$ follows the
  ordinary generating function
  $$\sum_{n \geq 0} \qfact{n}{q} x^n.$$
  
  Take $\F$ with subsets of $q$-\features{} marked by some variable
  $w$. As a $q$-\feature{} is either marked or it is not then the
  generating function for such structures is given by the substitution
  of $q$ by $w+1$ into the previous equation. We therefore have
  $$\sum_{n \geq 0} \qfact{n}{w+1} x^n.$$
  
  We now use that there exists a bijection between $\F$ structures
  of size $n$ with $i$ marked $q$-\features{} and $\F$ structures of
  size $n+i$ with $i$ marked $p$-\features{}.  This allows that
  subsets of $q$-\features{} marked with $w$ can be taken to subsets
  of $p$-\features{} marked by $z$ with the inclusion of an
  additional element. In terms of generating function this
  corresponds to the substitution of $w$ by $xz$.
  
  Therefore the ordinary generating function of $\F$ structures with
  subsets of marked $p$-features is 
  $$\sum_{n \geq 0} \qfact{n}{xz+1}x^n.$$
  
  If subsets of $p$-\features{} are marked, then each $\F$ structure
  occurs $i \choose j$ times with $j$ marked $p$-\features.  By the
  sieve principle (see Wilf ~\cite[Chapter 4, Section
    2]{generatingfunctionology}), as in the previous corollary,
  through the substitution of z by $y-1$ it then follows that $\F$
  structures with \emph{all} $p$-\features{} marked are given by the
  ordinary generating function for the Fishburn distribution
  $$\sum_{n \geq 0} \qfact{n}{x(y-1)+1}x^n.$$
\end{proof}

  \newcommand{\confused}{confused}
\newcommand{\unconfuse}{unconfuse}
\section*{Zero alignment matchings}

In this section we apply Theorem
\ref{theorem-q-factorial-construction} to identify a new Fishburn
statistic on a subset of matchings. Explicitly, matchings with zero
alignments follow the Fishburn distribution according to the number of
arcs which are both left-nesting and right-crossed.

Recall that two arcs $(i,j)$ and $(k,\ell)$ are an alignment if $i < j
< k < \ell$.

A matching with no alignments (a zero alignment matching) is
equivalently characterized as one where \emph{all} the openers in the
diagrammatic representation occur before \emph{all} the closers.

\begin{proposition}
 There are $n!$ zero alignment matchings of semi-length $n$.

 Furthermore they are enumerated by $\qfact{n}{}$ when refined
 according to the number of nestings.
\end{proposition}
\begin{proof}
  This is easiest seen via recursion with a bijection between
  matchings with no alignments and inversion tables.  Take the empty
  matching and the empty inversion table to be in bijection.

  Let $b_1 b_2 \dots b_n$ be an inversion table with each $b_i \in
  [0,i-1]$ and $M$ the matching constructed from $b_1 b_2 \dots
  b_{n-1}$. Label the position to the left of the first closer as $0$
  and label the positions to the left of an opener right-to-left from
  $1$ to $n-1$. Insert a new arc into $M$ with opener at position
  $b_n$ and closer at the rightmost position in the matching.

  By construction inserted openers occur to the left of all the
  closers and it is easy to see that entries in the inversion table
  correspond to the number of nested arcs.
\end{proof}

Recall that a left-nesting arc is an arc $(i,j)$ such that there
exists an arc $(i+1, \ell)$ with $\ell < j$. Recall also that $(i,j)$
is right-crossed if there exists an arc $(k,j-1)$ with $k < i$. We
shall call an arc which is both left-nesting and right-crossed a
\emph{\confused{}} arc.

Define an order on embraced nested openers. We shall write embraced
nested openers as ordered pairs. Take $((i,j),k)$ and $((i',j'),k')$
where $k$ and $k'$ are closers with $(i,j)$ an arc embracing $k$ and
$(i',j')$ an arc embracing $k'$. If $k=k'$ then, without loss of
generality, assume $j<j'$ and define
$$((i',j'),k') < ((i,j),k).$$ 
Otherwise, without loss of generality, assume $k<k'$ then define
$$((i',j'),k') < ((i,j),k).$$ 

For example, take the matching $\{(1,9),(2,12),(3,10),(4,7),(5,8),(6,11)\}$.
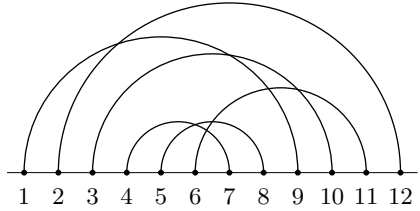
  \begin{center} 
    \begin{tikzpicture}[line width=.5pt,scale=0.45, font=\small]
      \filldraw (1,0) circle (1.9pt); \node at (1,-0.7){1};
      \filldraw (2,0) circle (1.9pt); \node at (2,-0.7){2};
      \filldraw (3,0) circle (1.9pt); \node at (3,-0.7){3};
      \filldraw (4,0) circle (1.9pt); \node at (4,-0.7){4};
      \filldraw (5,0) circle (1.9pt); \node at (5,-0.7){5};
      \filldraw (6,0) circle (1.9pt); \node at (6,-0.7){6};
      \filldraw (7,0) circle (1.9pt); \node at (7,-0.7){7};
      \filldraw (8,0) circle (1.9pt); \node at (8,-0.7){8};
      \filldraw (9,0) circle (1.9pt); \node at (9,-0.7){9};
      \filldraw (10,0) circle (1.9pt); \node at (10,-0.7){10};
      \filldraw (11,0) circle (1.9pt); \node at (11,-0.7){11};
      \filldraw (12,0) circle (1.9pt); \node at (12,-0.7){12};
      \draw[very thin] (0.5,0) -- (12.500000,0);
      \draw (9,0) arc (0:180:4.0);
      \draw (12,0) arc (0:180:5.0);
      \draw (10,0) arc (0:180:3.5);
      \draw (7,0) arc (0:180:1.5);
      \draw (8,0) arc (0:180:1.5);
      \draw (11,0) arc (0:180:2.5);
    \end{tikzpicture}
  \end{center}

The following subset of embraced nested openers are sorted:
$$  ((2,12),4) \quad < \quad  ((1,9),4)  \quad < \quad ((2,12),3).$$

Given a matching with a subset of embraced openers marked, using the
above order, for each embraced nested opener $((i,j),k)$ insert a new
arc opening immediately to the left of the embraced nested opener $k$
and closing immediately to the right of arc closer $j$. As $i<k$ the
new arc is therefore right-crossed, furthermore as the arc with opener
$k$ is nested by $(i,j)$ it follows that the newly inserted arc left
nests the arc with opener $k$. As both right-crossed and left-nesting
the inserted arc is \confused{}.

\begin{example}
  We demonstrate on our example matching.
  
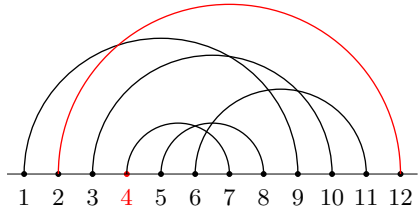
\begin{center} 
    \begin{tikzpicture}[line width=.5pt,scale=0.45, font=\small]
      \filldraw (1,0) circle (1.9pt); \node at (1,-0.7){1};
      \filldraw (2,0) circle (1.9pt); \node at (2,-0.7){2};
      \filldraw (3,0) circle (1.9pt); \node at (3,-0.7){3};
      {\red \filldraw (4,0) circle (1.9pt); \node at (4,-0.7){4};}
      \filldraw (5,0) circle (1.9pt); \node at (5,-0.7){5};
      \filldraw (6,0) circle (1.9pt); \node at (6,-0.7){6};
      \filldraw (7,0) circle (1.9pt); \node at (7,-0.7){7};
      \filldraw (8,0) circle (1.9pt); \node at (8,-0.7){8};
      \filldraw (9,0) circle (1.9pt); \node at (9,-0.7){9};
      \filldraw (10,0) circle (1.9pt); \node at (10,-0.7){10};
      \filldraw (11,0) circle (1.9pt); \node at (11,-0.7){11};
      \filldraw (12,0) circle (1.9pt); \node at (12,-0.7){12};
      \draw[very thin] (0.5,0) -- (12.500000,0);
      \draw (9,0) arc (0:180:4.0);
      {\red \draw (12,0) arc (0:180:5.0);}
      \draw (10,0) arc (0:180:3.5);
      \draw (7,0) arc (0:180:1.5);
      \draw (8,0) arc (0:180:1.5);
      \draw (11,0) arc (0:180:2.5);
    \end{tikzpicture}
  \end{center}

  Consider the following nested openers marked.
  $$  ((2,12),4) \quad  ((1,9),4)  \quad  ((2,12),3)$$

  As in the example for permutations the next nested opener to be
  considered will be colored red and inserted arcs blue. Inserting a
  confused arc from the first embraced nested opener results in the
  following matching.

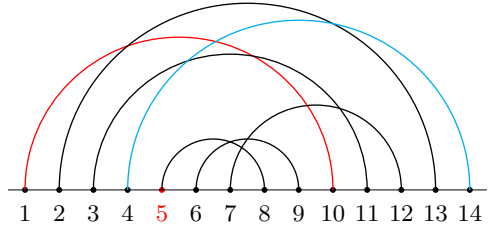
\begin{center} 
\begin{tikzpicture}[line width=.5pt,scale=0.45, font=\small]
 \filldraw (1,0) circle (1.9pt); \node at (1,-0.7){1};
 \filldraw (2,0) circle (1.9pt); \node at (2,-0.7){2};
 \filldraw (3,0) circle (1.9pt); \node at (3,-0.7){3};
 \filldraw (4,0) circle (1.9pt); \node at (4,-0.7){4};
{\red \filldraw (5,0) circle (1.9pt); \node at (5,-0.7){5};}
 \filldraw (6,0) circle (1.9pt); \node at (6,-0.7){6};
 \filldraw (7,0) circle (1.9pt); \node at (7,-0.7){7};
 \filldraw (8,0) circle (1.9pt); \node at (8,-0.7){8};
 \filldraw (9,0) circle (1.9pt); \node at (9,-0.7){9};
 \filldraw (10,0) circle (1.9pt); \node at (10,-0.7){10};
 \filldraw (11,0) circle (1.9pt); \node at (11,-0.7){11};
 \filldraw (12,0) circle (1.9pt); \node at (12,-0.7){12};
 \filldraw (13,0) circle (1.9pt); \node at (13,-0.7){13};
 \filldraw (14,0) circle (1.9pt); \node at (14,-0.7){14};
\draw[very thin] (0.5,0) -- (14.500000,0);
 {\red \draw (10,0) arc (0:180:4.5);}
 \draw (13,0) arc (0:180:5.5);
 \draw (11,0) arc (0:180:4.0);
 {\blue \draw (14,0) arc (0:180:5.0);}
 \draw (8,0) arc (0:180:1.5);
 \draw (9,0) arc (0:180:1.5);
 \draw (12,0) arc (0:180:2.5);
\end{tikzpicture}
\end{center}

The next two steps are as follows.

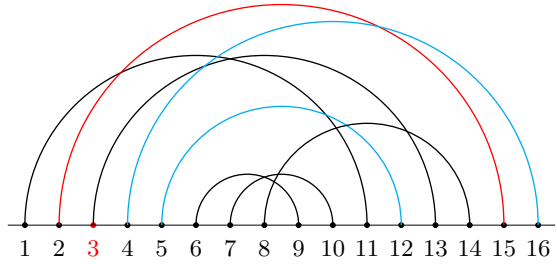
\begin{center}
\begin{tikzpicture}[line width=.5pt,scale=0.45, font=\small]
 \filldraw (1,0) circle (1.9pt); \node at (1,-0.7){1};
 \filldraw (2,0) circle (1.9pt); \node at (2,-0.7){2};
{\red  \filldraw (3,0) circle (1.9pt); \node at (3,-0.7){3};}
 \filldraw (4,0) circle (1.9pt); \node at (4,-0.7){4};
 \filldraw (5,0) circle (1.9pt); \node at (5,-0.7){5};
 \filldraw (6,0) circle (1.9pt); \node at (6,-0.7){6};
 \filldraw (7,0) circle (1.9pt); \node at (7,-0.7){7};
 \filldraw (8,0) circle (1.9pt); \node at (8,-0.7){8};
 \filldraw (9,0) circle (1.9pt); \node at (9,-0.7){9};
 \filldraw (10,0) circle (1.9pt); \node at (10,-0.7){10};
 \filldraw (11,0) circle (1.9pt); \node at (11,-0.7){11};
 \filldraw (12,0) circle (1.9pt); \node at (12,-0.7){12};
 \filldraw (13,0) circle (1.9pt); \node at (13,-0.7){13};
 \filldraw (14,0) circle (1.9pt); \node at (14,-0.7){14};
 \filldraw (15,0) circle (1.9pt); \node at (15,-0.7){15};
 \filldraw (16,0) circle (1.9pt); \node at (16,-0.7){16};
\draw[very thin] (0.5,0) -- (16.500000,0);
 \draw (11,0) arc (0:180:5.0);
{\red \draw (15,0) arc (0:180:6.5);}
 \draw (13,0) arc (0:180:5.0);
{\blue \draw (16,0) arc (0:180:6.0);}
{\blue \draw (12,0) arc (0:180:3.5);}
 \draw (9,0) arc (0:180:1.5);
 \draw (10,0) arc (0:180:1.5);
 \draw (14,0) arc (0:180:3.0);
\end{tikzpicture}

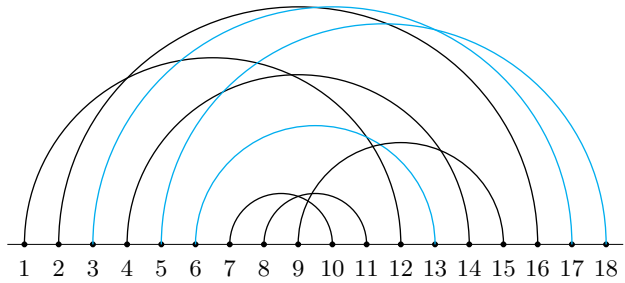
\begin{tikzpicture}[line width=.5pt,scale=0.45, font=\small]
 \filldraw (1,0) circle (1.9pt); \node at (1,-0.7){1};
 \filldraw (2,0) circle (1.9pt); \node at (2,-0.7){2};
 \filldraw (3,0) circle (1.9pt); \node at (3,-0.7){3};
 \filldraw (4,0) circle (1.9pt); \node at (4,-0.7){4};
 \filldraw (5,0) circle (1.9pt); \node at (5,-0.7){5};
 \filldraw (6,0) circle (1.9pt); \node at (6,-0.7){6};
 \filldraw (7,0) circle (1.9pt); \node at (7,-0.7){7};
 \filldraw (8,0) circle (1.9pt); \node at (8,-0.7){8};
 \filldraw (9,0) circle (1.9pt); \node at (9,-0.7){9};
 \filldraw (10,0) circle (1.9pt); \node at (10,-0.7){10};
 \filldraw (11,0) circle (1.9pt); \node at (11,-0.7){11};
 \filldraw (12,0) circle (1.9pt); \node at (12,-0.7){12};
 \filldraw (13,0) circle (1.9pt); \node at (13,-0.7){13};
 \filldraw (14,0) circle (1.9pt); \node at (14,-0.7){14};
 \filldraw (15,0) circle (1.9pt); \node at (15,-0.7){15};
 \filldraw (16,0) circle (1.9pt); \node at (16,-0.7){16};
 \filldraw (17,0) circle (1.9pt); \node at (17,-0.7){17};
 \filldraw (18,0) circle (1.9pt); \node at (18,-0.7){18};
\draw[very thin] (0.5,0) -- (18.500000,0);
 \draw (12,0) arc (0:180:5.5);
 \draw (16,0) arc (0:180:7.0);
{\blue \draw (17,0) arc (0:180:7.0);}
 \draw (14,0) arc (0:180:5.0);
{\blue \draw (18,0) arc (0:180:6.5);}
{\blue \draw (13,0) arc (0:180:3.5);}
 \draw (10,0) arc (0:180:1.5);
 \draw (11,0) arc (0:180:1.5);
 \draw (15,0) arc (0:180:3.0);
\end{tikzpicture}
\end{center}

It is easily checked that the inserted arcs are \confused{} and their
removal returns the original matching.
\end{example}

\begin{proposition}
  The above is a bijection between zero alignment matchings with
  marked embraced openers and zero alignment matchings with marked
  \confused{} arcs.
\end{proposition}
\begin{proof}
  We are required to show that at each stage the process is well
  defined: that no alignments are introduced and and that no
  previously inserted \confused{} arc has its left nesting or right
  crossed attributes removed.

  That no alignment is introduced can be seen by contradiction. As the
  inserted arc has its opener to the left of an existing opener and
  its closer to the right of an existing closer no new alignment can
  be introduced if the original matching was a zero alignment
  matching.

  That each step of the process does not break the right-crossed or
  left-nesting property of a previously  inserted arc
  is given by the order on nested openers. If two inserted arcs share
  the same opener as part of their nested opener, then that both arcs
  are still left nesting is given by the order on nesting arc
  closers. If two inserted arcs share the same nesting arc closer as
  part of their nested opener, then that both arcs are right nesting
  is given by the order of the opener.

  Each inserted arc has its opener and closer uniquely determined by
  the nested opener. Furthermore it is clear that removing inserted
  arcs returns the original matching. Injectivity and surjectivy are
  thus simple. 
\end{proof}

The following corollary then results from Theorem
\ref{theorem-q-factorial-construction} and the above Proposition.

\begin{corollary}
 Zero alignment matchings with respect to \confused{} arcs follow the
 Fishburn distribution.
\end{corollary}

  \section{Factorial posets}

Identifying appropriate statistics and applying the technique given by
Theorem \ref{theorem-q-factorial-construction} to the factorial posets
of Claesson and Linusson~\cite{factorial-matchings} allows for a new
method for the enumeration of interval orders which differs from both
the recursive construction of Bousquet-M\'elou
\etal~\cite{two-plus-two-free} and the matrix hook bijection of Dukes,
Jel\'inek and Kubitzke~\cite{comp-matrices-spec}.

Claesson and Linusson~\cite{factorial-matchings} define the
\emph{factorial posets}, a set of labeled interval orders counted by
$n!$, as follows. A factorial poset $P$ on some linearly ordered
underlying set $U$ is a naturally labeled poset with the additional
condition that, for $i,j,k \in U$,
$$i <_U j <_P k \implies i <_P k.$$
This is referred to as the \emph{factorial condition}.

Easily seen to be equivalent, a poset is factorial if and only if for
all $k \in P$ there exists $j \in [0,k-1]$ such that $\Pre{k} =
[1,j]$. As the predecessor sets can be linearly ordered by inclusion
it follows that factorial posets are interval orders.

Claesson and Linusson take advantage of this by using entries of an
inversion table to encode the construction of a factorial poset, thus
giving that the two structures are in bijection.  We include their
result for completeness.

\begin{theorem}[Claesson and Linusson~\cite{factorial-matchings}]\label{thm-fact-poset}
  Factorial posets on $[n]$ are in bijection with inversion tables of
  length $n$.
\end{theorem}
\begin{proof}
As a poset is factorial if and only if for all $k \in P$ there exists
$j \in [0,k-1]$ such that $\Pre{k} = [1,j]$. An inversion table $b_1
b_2 \dots b_n$ is given by setting $b_k$ to the value $j \in [0,k-1]$.
\end{proof}

Claesson and Linusson identify numerous statistics preserved by their
bijection. In particular that the number of \emph{incomparable pairs} in
factorial posets, defined as
$$ | \{ (i,j) \in P \times P : i \not <_P j, i <_U j \}|,$$ are
enumerated by the $q$-factorial. 

Taking two factorial posets to be equivalent if they are structurally
isomorphic, Claesson and Linusson demonstrate that posets satisfying
that for all $i \in [n-1]$
$$\pre{i} \leq \pre{(i+1)} \quad \mbox{or} \quad \suc{i} > \suc{(i+1)}$$
are unique representatives of their class. 

Again we include their result.
\begin{proposition}[Claesson and Linusson~\cite{factorial-matchings}]
\label{prop-claesson-linusson}
There is exactly one way to label a $(2+2)$-free poset such that it
satisfies 
$$\pre{i} \leq \pre{(i+1)} \quad \mbox{or} \quad \suc{i} > \suc{(i+1)}.$$
\end{proposition}
\begin{proof}
A poset satisfying the above condition has that for all $i\in [n]$
the pairs
$$(\suc{i},\pre{i}) $$
are weakly decreasing on the first coordinate and weakly increasing on the
second. The factorial condition gives that for $i,j \in [n]$ the pairs
$(\suc{i},\pre{i})$ and $(\suc{j},\pre{j})$ are equal if an only if
the pairs are indistinguishable within the poset, thus giving a
canonical labeling.
\end{proof} 

We extend this notion to consider a new feature on factorial posets,
explicitly elements which \emph{fail} to satisfy this property.

\begin{definition}[Mislabeling]
Define a \mislabeling{} in a factorial poset on $[n]$ to be an $i \in
[n-1]$ such that
$$\pre{i} > \pre{(i+1)} \quad \mbox{and} \quad \suc{i} \leq \suc{(i+1)}.$$
\end{definition}

\begin{example}
 The poset
  \begin{center}
  \begin{tikzpicture}[xscale=0.5, yscale=0.7, semithick, baseline=12]
    \style;
    \node [wht] (a) at (-1,0) {};
    \node [wht] (b) at ( 0,0) {};
    \node [wht] (c) at ( 1,0) {};
    \node [wht] (d) at ( 2,0) {};
    \node [wht] (e) at (-1,1) {};
    \node [wht] (f) at ( 0,2) {};
    \draw (a) node[below=2pt] {1} -- 
          (e) node[left=2pt]  {\blue{2}} -- 
          (f) node[right=2pt] {\blue{4}};
    \draw (b) node[below=2pt] {3} -- (f);
    \draw (c) node[below=2pt] {5};
    \draw (d) node[below=2pt] {6};
  \end{tikzpicture}
  \end{center}
  
  Has set of \mislabelings{} $\{2,4\}$.
\end{example}

By definition a factorial poset with zero \mislabelings{} satisfies the
condition from Proposition \ref{prop-claesson-linusson} and is thus a
unique representative of its isomorphism class.

A consequence of the factorial condition is that if $\pre{i}>
\pre{(i+1)}$ then $i$ and $i+1$ are incomparable as if $i <_P i+1$
then the factorial condition requires that for all $\ell <_U i
\implies \ell <_P {i+1}$.  Furthermore there must exist $\ell$ such
that $ \ell <_P i$ but that $\ell \not <_P i+1$.

Therefore an equivalent condition to $\pre{i}> \pre{(i+1)}$ is that
there exists an induced subposet isomorphic to $(2+1)$ with the
following labeling
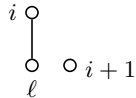
\begin{center}
  \begin{tikzpicture}[xscale=0.5, yscale=0.7, semithick, baseline=12]
    \style;
    \node [wht] (a) at (-1,0) {};
    \node [wht] (b) at ( 0,0) {};
    \node [wht] (c) at (-1,1) {};
    \draw (a) node[below=2pt] {$\ell$} -- 
          (c) node[left=2pt]  {$i$};
    \draw (b) node[below=2pt,right=2pt] {$i+1$};
  \end{tikzpicture}
\end{center}

\subsection{Sieved enumeration of interval orders}

Recall that we write incomparable pairs as $(i,j)$ with $i <_U j$ and
let $U$ be some linearly ordered set with $|U| = n$. For some $k \in
[0, n-1]$ take a poset $P$ built on the first $n-k$ elements of $U$
with $k$ marked incomparable pairs.

Define an order on incomparable pairs. Let $(i,j)$ and $(i',j')$ be
two pairs. If $j=j'$ without loss of generality assume $i <_U
i'$. Then we define
$$(i',j') < (i,j).$$
Otherwise $j \neq j'$ then without loss of generality assume
$j <_U j'$. Then we define
$$(i,j) < (i',j').$$

To illustrate, the following pairs are sorted according to the above
order.
$$(2,3)<(1,3)<(4,6)<(3,6).$$

In this order, each pair $(i,j)$ is then used to insert a new element
into the poset. This new element has predecessor set
$$ \{h \in P : h \leq_U i \},$$
and successor set
$$\Suc{j}.$$
Increment all $k \in P$ with $k \geq_U j$ to its immediate
successor in $U$, giving the newly inserted element the value $j$.

By definition this introduces an occurrence of 
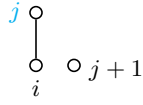
\begin{center}
  \begin{tikzpicture}[xscale=0.5, yscale=0.7, semithick, baseline=12]
    \style;
    \node [wht] (a) at (-1,0) {};
    \node [wht] (b) at ( 0,0) {};
    \node [wht] (c) at (-1,1) {};
    \draw (a) node[below=2pt] {$i$} -- 
          (c) node[left=2pt]  {\blue{$j$}};
    \draw (b) node[below=2pt,right=2pt] {$j+1$};
  \end{tikzpicture}
\end{center}
into the new poset with the inserted element marked in
blue. Furthermore as the successor set of the inserted element $j$ is
equal to that of the successor set of the element now labeled $j+1$ it
follows that the newly inserted element with label $j$ is a
\mislabeling{}.

\begin{example}

  Consider our earlier factorial poset built on $[6]$.

  
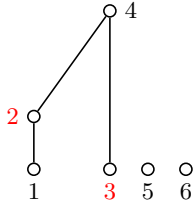
\begin{center}
  \begin{tikzpicture}[xscale=0.5, yscale=0.7, semithick, baseline=12]
    \style;
    \node [wht] (a) at (-2,0) {};
    \node [wht] (b) at ( 0,0) {};
    \node [wht] (c) at ( 1,0) {};
    \node [wht] (d) at ( 2,0) {};
    \node [wht] (e) at (-2,1) {};
    \node [wht] (f) at ( 0,3) {};
    \draw (a) node[below=2pt] {1} -- 
          (e) node[left=2pt]  {\red{2}} -- 
          (f) node[right=2pt] {4};
    \draw (b) node[below=2pt] {\red{3}} -- (f);
    \draw (c) node[below=2pt] {5};
    \draw (d) node[below=2pt] {6};
  \end{tikzpicture}
  \end{center}

  We consider the following incomparable pairs to be marked 
  $$(2,3)<(1,3)<(4,6)<(3,6)$$

  As the values within the pairs change at each stage we will denote
  the next incomparable pair to be used in red. Inserted elements will
  be colored blue.

  The pair $(2,3)$ specifies the new element $\ell$ to be inserted
  defined by
  $$\Pre{\ell} = \{ h \in U : h \leq_U 2 \} = \{1,2 \} \quad
  \mbox{and} \quad \Suc{\ell} = \Suc{3} = \{ 4\}.
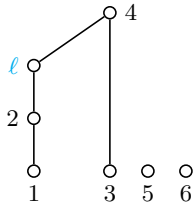
$$

  \begin{center}
  \begin{tikzpicture}[xscale=0.5, yscale=0.7, semithick, baseline=12]
    \style;
    \node [wht] (a) at (-2,0) {};
    \node [wht] (b) at ( 0,0) {};
    \node [wht] (c) at ( 1,0) {};
    \node [wht] (d) at ( 2,0) {};
    \node [wht] (e) at (-2,1) {};
    \node [wht] (f) at ( 0,3) {};

    \node [wht] (g) at (-2,2) {};

    \draw (a) node[below=2pt] {1} -- 
          (e) node[left=2pt]  {2} -- 
          (g) node[left=2pt]  {\blue{$\ell$}} --
          (f) node[right=2pt] {4};
    \draw (b) node[below=2pt] {3} -- (f);
    \draw (c) node[below=2pt] {5};
    \draw (d) node[below=2pt] {6};
  \end{tikzpicture}
  \end{center}

  All elements with label greater or equal to $3$ are incremented by
  one and the newly inserted element $\ell$ is given the label $3$.
  
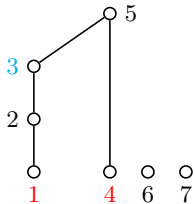
\begin{center}
  \begin{tikzpicture}[xscale=0.5, yscale=0.7, semithick, baseline=12]
    \style;
    \node [wht] (a) at (-2,0) {};
    \node [wht] (b) at ( 0,0) {};
    \node [wht] (c) at ( 1,0) {};
    \node [wht] (d) at ( 2,0) {};
    \node [wht] (e) at (-2,1) {};
    \node [wht] (f) at ( 0,3) {};

    \node [wht] (g) at (-2,2) {};

    \draw (a) node[below=2pt] {\red{1}} -- 
          (e) node[left=2pt]  {2} -- 
          (g) node[left=2pt]  {\blue{3}} --
          (f) node[right=2pt] {5};
    \draw (b) node[below=2pt] {\red{4}} -- (f);
    \draw (c) node[below=2pt] {6};
    \draw (d) node[below=2pt] {7};
  \end{tikzpicture}
  \end{center}

The remaining steps are as follows.

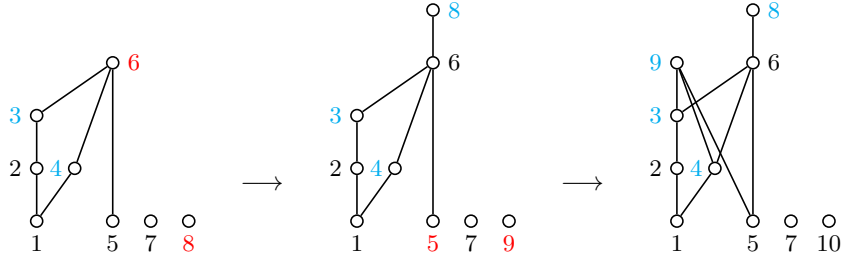
\begin{center}
  \begin{tikzpicture}[xscale=0.5, yscale=0.7, semithick, baseline=12]
    \style;
    \node [wht] (a) at (-2,0) {};
    \node [wht] (b) at ( 0,0) {};
    \node [wht] (c) at ( 1,0) {};
    \node [wht] (d) at ( 2,0) {};
    \node [wht] (e) at (-2,1) {};
    \node [wht] (f) at ( 0,3) {};

    \node [wht] (g) at (-2,2) {};
    \node [wht] (h) at (-1,1) {};

    \draw (a) node[below=2pt] {1} -- 
          (e) node[left=2pt]  {2} -- 
          (g) node[left=2pt]  {\blue{3}} --
          (f) node[right=2pt] {\red{6}};
    \draw (b) node[below=2pt] {5} -- (f);
    \draw (c) node[below=2pt] {7};
    \draw (d) node[below=2pt] {\red{8}};

    \draw (a) -- (h) node[left=1pt] {\blue{4}} -- (f);
  \end{tikzpicture}
  \hspace{0.25cm} $\longrightarrow$ \hspace{0.25cm}
  \begin{tikzpicture}[xscale=0.5, yscale=0.7, semithick, baseline=12]
    \style;
    \node [wht] (a) at (-2,0) {};
    \node [wht] (b) at ( 0,0) {};
    \node [wht] (c) at ( 1,0) {};
    \node [wht] (d) at ( 2,0) {};
    \node [wht] (e) at (-2,1) {};
    \node [wht] (f) at ( 0,3) {};

    \node [wht] (g) at (-2,2) {};
    \node [wht] (h) at (-1,1) {};
    \node [wht] (i) at (0,4) {};

    \draw (a) node[below=2pt] {1} -- 
          (e) node[left=2pt]  {2} -- 
          (g) node[left=2pt]  {\blue{3}} --
          (f) node[right=2pt] {6};
    \draw (b) node[below=2pt] {\red{5}} -- (f);
    \draw (c) node[below=2pt] {7};
    \draw (d) node[below=2pt] {\red{9}};

    \draw (a) -- (h) node[left=1pt] {\blue{4}} -- (f);
    \draw (f) -- (i) node[right=2pt] {\blue{8}};
  \end{tikzpicture}
  \hspace{0.25cm} $\longrightarrow$ \hspace{0.25cm}
  \begin{tikzpicture}[xscale=0.5, yscale=0.7, semithick, baseline=12]
    \style;
    \node [wht] (a) at (-2,0) {};
    \node [wht] (b) at ( 0,0) {};
    \node [wht] (c) at ( 1,0) {};
    \node [wht] (d) at ( 2,0) {};
    \node [wht] (e) at (-2,1) {};
    \node [wht] (f) at ( 0,3) {};

    \node [wht] (g) at (-2,2) {};
    \node [wht] (h) at (-1,1) {};
    \node [wht] (i) at (0,4) {};
    \node [wht] (j) at (-2,3) {};

    \draw (a) node[below=2pt] {1} -- 
          (e) node[left=2pt]  {2} -- 
          (g) node[left=2pt]  {\blue{3}} --
          (f) node[right=2pt] {6};
    \draw (b) node[below=2pt] {5} -- (f);
    \draw (c) node[below=2pt] {7};
    \draw (d) node[below=2pt] {10};

    \draw (a) -- (h) node[left=1pt] {\blue{4}} -- (f);
    \draw (f) -- (i) node[right=2pt] {\blue{8}};
   
    \draw (h) -- (j) node[left=2pt] {\blue{9}} -- (g);
    \draw (j) -- (b);
  \end{tikzpicture}
  \end{center}

  Thus we have returned a poset of size $10$ with set of marked
  \mislabelings{} $\{3,4,8,9\}$.

\end{example}
\begin{proposition}\label{proposition-factorial-poset-bijection}
  The procedure described above gives a bijection between factorial
  posets with marked incomparable pairs and factorial posets with
  marked \mislabelings{}.
\end{proposition}
\begin{proof}

  We shall first show that the process described above is well
  defined. This is equivalent to showing that at each insertion the
  following properties are preserved for the resulting poset: it is
  naturally labeled, it satisfies the factorial condition and it does
  not remove any \mislabelings{} previously inserted by the
  process. 

  For the incomparable pair $(i,j)$ and newly inserted element $\ell$
  we have that $\ell$ covers all $\{ h \in P: h <_U i\}$ and thus by
  construction it is both naturally labeled and satisfies the
  factorial condition. Elements smaller than $j$ under $U$ remain
  unchanged. The newly inserted element is given the same successor
  set as the element which previously had that label and thus the
  insertion does not break the factorial condition for any element
  larger under $U$ than $j$ and also ensures that the naturally
  labeled property is preserved.

  That no previously inserted \mislabelings{} are removed by the
  process is given by the order on incomparable pairs and the
  factorial property, thus the process is a mapping between factorial
  posets with marked incomparable pairs to factorial posets with
  marked \mislabelings{}.


  Next we show that the mapping is bijective. That it is injective
  follows from the total order defined on incomparable pairs and that
  the insertion of a new element is uniquely determined by an
  incomparable pair.

  It remains to show surjectivity. The process we have defined inserts
  \mislabelings{} in order according to $U$. We can consider the
  reverse of the insertion operation taking a factorial poset with
  marked \mislabelings{} to a factorial poset with marked incomparable
  pairs.

  Given a factorial poset of size $n$ with $k$ marked \mislabelings{}
  take the \mislabeling{} with the largest value $j$ and remove it
  from the poset.  As $j$ is a mislabeling there exists $\ell <_P j$
  such that $\ell \not < j+1$. Take the largest such $\ell$ and mark
  the incomparable pair consisting of $(\ell, j+1)$. Thus we have
  returned a poset of size $n-1$ with $i-1$ mislabelings and $1$
  marked incomparable pair. Surjectivity follows from repeated
  application.

 As it is both surjective and injective the mapping is a bijection.
\end{proof} 

The following corollary then results from Theorem
\ref{theorem-q-factorial-construction} and Proposition
\ref{proposition-factorial-poset-bijection}.

\begin{corollary}
 Factorial posets follow the Fishburn distribution according to the number of \mislabelings{}.
\end{corollary}

Substitution of $y=0$ into the ordinary generating function for the
Fishburn distribution,
  $$\sum_{n \geq 0} \qfact{n}{x(y-1)+1}x^n|_{y=0},$$ 
returns the ordinary
generating function for factorial posets with no \mislabelings{}.
Proposition \ref{prop-claesson-linusson} gives that such posets are
unique representatives of their isomorphism class thus yielding, as
expected, the result of Bousquet-M\'{e}lou
\etal~\cite{two-plus-two-free} that the generating function for
unlabeled interval orders is given by
 $$\sum_{n \geq 0} \qfact{n}{-x+1}x^n.$$

  \section{Fishburn distribution}

\begin{figure}
  \begin{center}
    1\\ 2\\ 6, 1\\ 24, 9\\ 120, 72, 5\\ 720, 600, 98, 1\\ 5040, 5400,
    1450, 76\\ 40320, 52920, 20100, 2200, 35\\ 362880, 564480, 279300,
    48750, 2299, 9\\ \dots \\
    \caption{Unsieved Fishburn distribution: number of structures of
      size $n$ with $i$ marked $p$-features,
      $u_{n,i}$}\label{fig-unsieved}
  \end{center}
\end{figure}

We obtain the following corollaries concerning the Fishburn
distribution from Theorem \ref{theorem-q-factorial-construction} and
its proof.
\begin{corollary}
  For some appropriate structure let $p$ be a \feature{} which follows
  the Fishburn distribution and $q$ a \feature{} which follows the
  Mahonian distribution.

  Letting $u_{n,i}$ denote the number of structures of size $n$ with
  $i$ marked $p$-\features{},
  $$\sum_{n \geq 0} \sum_{i \geq 0} u_{n,i} x^n z^i = \sum_{n \geq 0}
  \qfact{n}{xz+1} x^n,$$ we have that
  $$
    u_{n,i} = \sum_{j= i}^{n-i \choose 2} {j \choose i} m_{n-i,j}.
  $$
\end{corollary}
The first few terms of $u_{n,i}$ are shown in Figure
\ref{fig-unsieved}
\begin{proof}
  Theorem \ref{theorem-q-factorial-construction} gives that a
  $q$-factorial structure of size $n-i$ with $i$ marked
  $q$-\features{} can be extended to a structure of size $n$ with $i$
  marked $p$-\features{}.

  For a Mahonian structure of size $n-i$ with $j$ $q$-features then
  $i$ are selected. The number of $q$-factorial structures of size
  $n-i$ with $j$ $q$-features is given by Mahonian number
  $m_{n-i,j}$.

  The maximum number of $q$-\features{} a $q$-factorial structure of
  size $n-i$ can have is ${n-i \choose 2}$. Thus $j$ is bounded
  $$i \leq j \leq {n-i \choose j}.$$ 
\end{proof}

\begin{remark}
The row sums of Figure \ref{fig-unsieved} (\oeis{A179525}), i.e.\
$$ \sum_{i=0} u_{n,i},$$

have previously been studied by Jel{\'i}nek~\cite{jelinek-self-dual}
as counting \emph{primitive row Fishburn matrices}, upper-triangular,
binary non-row empty matrices, according to the sum of the entries.
Jel{\'i}nek considers such matrices as part of his work on counting
self-dual interval orders; he demonstrates a relation
between the generating functions of self-dual interval orders
enumerated by a reduced size function and primitive row Fishburn
matrices. 

We note that the coefficient of $x^nz^k$ in the refined formula
$$\sum_{n \geq 0} \qfact{n}{xz+1}x^n$$ 

can be interpreted as counting the number of primitive row Fishburn
matrices such that:
\begin{enumerate}
\item There are a total of $k$ entries in the matrix that are not the
  first to occur in their row.
\item The entries in the matrix sum to $n$.
\end{enumerate}
\end{remark}

\begin{corollary}
  Recalling that $f_{n,k}$ denotes the coefficient in the Fishburn
  distribution
  $$\sum_{n \geq 0} \sum_{k \geq 0} f_{n,k} x^n y^k = \sum_{n \geq 0}
  \qfact{n}{x(y-1)+1} x^n,$$ we have that
  \begin{align*}
    f_{n,k} &= \sum_{i=k}^{n-2} (-1)^{i+k} {i \choose k} u_{n,i} \\ & =
    \sum_{i=k}^{n-2} (-1)^{i+k} {i \choose k} \sum_{j= i}^{n-i \choose 2}
        {j \choose i} m_{n-i,j}.
  \end{align*}
\end{corollary}

The first few terms are shown in Figure \ref{fig-sieved}
\begin{proof}

  Again for some appropriate structure let $p$ be a \feature{} which
  follows the Fishburn distribution.

  Recall from the proof of Theorem
  \ref{theorem-q-factorial-construction} that to take structures with
  subsets of $p$-\features{} marked to structures with all \emph{all}
  $p$-\features{} marked corresponds to the substitution of $y-1$ by
  $z$.

  The result then follows from the previous corollary and binomial
  expansion. 
\end{proof}

\begin{figure}
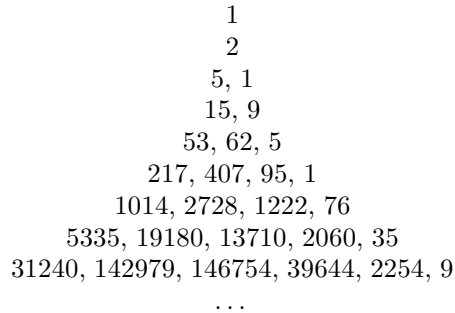

  \begin{center}
    1\\ 2\\ 5, 1\\ 15, 9\\ 53, 62, 5\\ 217, 407, 95, 1\\ 1014, 2728,
    1222, 76\\ 5335, 19180, 13710, 2060, 35\\ 31240, 142979, 146754,
    39644, 2254, 9\\ \dots \\
    \caption{Fishburn distribution, $f_{n,k}$}\label{fig-sieved}
  \end{center}
\end{figure}

\begin{remark}
In the above corollary the upper bound $n-2$ for the initial summation
is justified by noting that an occurrence of $\meshpata$ can be
uniquely determined by the first element in the occurrence and that
two more entries must follow in the permutation. Therefore there can
be no more than $n-2$ occurrences of a Fishburn statistic in a
structure of size $n$.

This is sufficient for our purposes however we note that this is
\emph{not} the least upper bound (easily checked empirically). We
leave this as an open question.
\end{remark}

\begin{question}
Is there an aesthetically pleasing expression for the least upper
bound for the value of a Fishburn statistic?
\end{question}

  \printbibliography{}

\end{document}